\documentclass[a4paper,oneside,11pt]{article}

\usepackage{amsxtra,amssymb,amsmath,amsthm,amsbsy,enumerate}
\usepackage[all,2cell,v2]{xy}\UseTwocells\UseHalfTwocells
\usepackage{tikz}
\usetikzlibrary{matrix, arrows,positioning}
\usepackage{hyperref}

\theoremstyle{definition}
\newtheorem{theorem}{Theorem}[section]
\newtheorem{lemma}[theorem]{Lemma}

\newtheorem{corollary}[theorem]{Corollary}
\newtheorem{example}[theorem]{Example}
\newtheorem{proposition}[theorem]{Proposition}

\def\id{{\rm id}}

\def\xto{\xrightarrow}
\def\to{\rightarrow}

\def\toto{\rightrightarrows}
\def\xfrom{\xleftarrow}
\def\from{\leftarrow}
\def\action{\curvearrowright}
\def\xto{\xrightarrow}
\def\to{\rightarrow}

\def\toto{\rightrightarrows}
\def\xfrom{\xleftarrow}
\def\from{\leftarrow}
\def\action{\curvearrowright}

\def\then{\Rightarrow}

\def\mon{{\rm Mon}}
\def\hol{{\rm Hol}}

\def\R{{\mathbb R}}

\def\bar{|}

\DeclareMathOperator{\colim}{colim}

\newcommand\flowsfrom{\mathrel{\reflectbox{$\leadsto$}}}

\begin{document}

\title{\bf On Hausdorff integrations of Lie algebroids}

\author{Matias del Hoyo \and Daniel L\'opez Garcia}

\date{}

\maketitle

\begin{abstract}
We present Hausdorff versions for Lie Integration Theorems 1 and 2 and apply them to study Hausdorff symplectic groupoids arising from Poisson manifolds. To prepare for these results we include a discussion on Lie equivalences and propose an algebraic approach to holonomy. We also include subsidiary results, such as a generalization of the integration of subalgebroids to the non-wide case, and explore in detail the case of foliation groupoids. 
\end{abstract}

\noindent{{\bf Keywords:} Lie groupoids, Poisson manifolds, Foliations\\
{\bf Math. Subj. Class.:} 22A22; 53D17; 57R30.}

\tableofcontents


\section{Introduction}

%


Lie groupoids and Lie algebroids play a central role in Higher Differential Geometry, serving as models for classic geometries such as actions, foliations and bundles, and with applications to Symplectic Geometry, Noncommutative Algebra and Mathematical Physics, among others \cite{dsw,mmbook}.
Every Lie groupoid yields a Lie algebroid through differentiation, setting a rich interplay between global and infinitesimal data, which is ruled by the so-called Lie Theorems. Lie 1 constructs a maximal (source-simply connected) Lie groupoid integrating the algebroid of a given groupoid, and Lie 2 shows that a Lie algebroid morphism can be integrated to a Lie groupoid morphism under a certain hypothesis \cite{mx,mmbook}. The hardest one, Lie 3, provides computable obstructions to the integrability of a Lie algebroid \cite{cf1}. 


When working with Lie groupoids one usually allows the manifold of arrows to be non-Hausdorff. One reason for that is to include the monodromy and holonomy groupoids arising from foliations.
Another reason is that the maximal groupoid given by Lie 1 may be non-Hausdorff even when the original groupoid is. 
A Poisson manifold yields a Lie algebroid on its cotangent bundle, which is integrable if and only if the Poisson manifold has a complete symplectic realization \cite{cf2}. 
The canonical symplectic structure on the cotangent bundle integrates to a symplectic structure on the source-simply connected groupoid \cite{cdw}, 
which may be non-Hausdorff, and smaller integrations may be Hausdorff but not symplectic. This note is motivated by the problem of understanding Hausdorff symplectic groupoids arising from Poisson manifolds.


Our first main result is a Hausdorff version for Lie 1, Theorem  \ref{thm:HLie1},  showing that every Hausdorff groupoid yields a maximal Hausdorff integration. We illustrate with examples that this may or may not agree with the source-simply connected integration or with the original one. In order to prove this theorem we pay special attention to Lie equivalences, namely Lie groupoid morphisms which are isomorphism at the infinitesimal level, and their characterization by their kernels, see Proposition \ref{prop:kernel}. This characterization appears for instance in \cite{gl} and \cite{mkbook}. We review it to set notations and to serve for quick reference. Our first result also relies on the construction of monodromy and holonomy groupoids, for which we include here an original algebraic approach, see Section \ref{sec:holonomy}, that the reader may find interesting on its own.



In analogy with the classic case, one might a priori expect that a Lie algebroid morphism can be integrated to a morphism between Hausdorff groupoids if the first one is maximal. We show with the Example \ref{ex:counter} that this is not the case. Our second main result is a Hausdorff version of Lie 2 that includes a holonomy hypothesis, Theorem \ref{thm:HLie2v1}. In order to prove this, we first develop a generalized version of the integration of Lie subalgebroids \cite{mm-s} that works in the non-wide case, in Proposition \ref{prop:subalgbd}, and which also allow us to complete a neat conceptual proof for classic Lie 2.  
We include a second Hausdorff version of Lie 2, Theorem \ref{thm:HLie2v2}, in the context of Lie groupoids arising from foliations, where the holonomy hypothesis becomes automatic, and which implies Corollary \ref{cor:mhfc}, the uniqueness of the maximal Hausdorff integration.


Finally, building on the theory of VB-groupoids and VB-algebroids \cite{pradines,mkbook,dla1,gsm,bcdh}, we apply our results to show Theorem \ref{thm:application}, stating that if the algebroid induced by a Poisson manifold is integrable by a Hausdorff groupoid, then the maximal Hausdorff integration is symplectic. We achieve this by looking at the canonical symplectic form on the cotangent bundle as a VB-algebroid morphism and then integrate it to a VB-groupoid morphism. This should be compared with \cite{mx}, where Lie 2 first appeared, as a device to integrate a bialgebroid to a Poisson groupoid in a similar fashion. Our application shows that the holonomy hypothesis can indeed be computable in concrete situations. Corollary \ref{cor:application} concludes that if a Poisson manifold is integrable by a Hausdorff groupoid then it admits a Hausdorff complete symplectic realization.


A Hausdorff version for Lie 3, namely the problem of deciding whether a Lie algebroid admits a Hausdorff integration at all, is left to be addressed elsewhere. Besides the usual obstructions to integrability, we know of examples of integrable algebroids which do not admit a Hausdorff integration, such as in the foliation described in Example \ref{ex:hol=mon}, and in the Lie algebra bundle constructed in \cite[VI.5]{dl}.
A similar question, whether the canonical (source-simply connected) integration is Hausdorff, is studied in \cite{m} for Poisson manifolds and in \cite{ach} for foliations.  
In this direction, we show in Corollary \ref{cor:ssch} that if there is a Hausdorff integration then the maximal integration is Hausdorff if and only if the foliation by source-fibers has no vanishing cycles. 




\medskip


{\bf Organization.} In section 2 we provide a systematic study of Lie equivalences, featuring the characterization by their kernels, and in section 3 we give the algebraic approach to monodromy and holonomy. These two sections do not contain new results and their originality, if any, lies in the developed approach. Section 4 reviews classic Lie 1, reformulated in a categorical way, and prove its Hausdorff version, our first main result, illustrated with several examples. In section 5 we generalize some of the results regarding the integration of Lie subalgebroids to the non-wide case, and we apply them in section 6, where we provide two versions for Hausdorff Lie 2. Finally, in section 7, we prove our last main result, a Hausdorff version of the integration of Poisson manifolds by Hausdorff symplectic groupoids. We include a proof of Godement criterion for quotients of non-Hausdorff smooth manifolds as an appendix.

\medskip


{\bf Notations and conventions.} We assume a certain familiarity with the basic theory of Lie groupoids and Lie algebroids, and refer to \cite{dsw,dh,mmbook} for further details. We denote a Lie groupoid either by $G\toto M$ or simply $G$, as the ambiguity between the whole groupoid and its manifold of arrows should be solved by the context. 
The object manifold $M$ is always Hausdorff, but $G$ need not be. 
Given $x\in M$ we write $G(-,x)$ for the source fiber and $G_x$ for the isotropy group. We write $A_G\then M$ or simply $A_G$ for the corresponding algebroid, and use the convention that $A_G=\ker ds\bar_M$.  If $\phi:G'\to G$ is a Lie groupoid morphism, we write $A_\phi$ for the map induced among the algebroids. The last section uses the notions of VB-groupoids and VB-algebroids following the treatment given in \cite{bcdh}.

\medskip

{\bf Acknowledgements.}
We thank H. Bursztyn, R. L. Fernandes and F. Bischoff for comments and suggestions on the first version of the paper. We also thank the anonymous referee for the thorough report which helped us to improve the article considerably.
MdH was partially supported by National Council for Scientific and Technological Development — CNPq grants 303034/2017-3 and 429879/2018-0, and by FAPERJ grant 210434/2019. 
DL was partially supported by PhD CNPq grant 140576/2017-7.

\section{Basics on Lie equivalences}


We review basic facts about morphisms inducing an isomorphism at the infinitesimal level. The results here are elementary and are scattered in the literature. We collect them to set notations and for quick reference throughout the paper.

\medskip


We say that a Lie groupoid morphism $\phi:G'\to G$ is a {\bf Lie equivalence} if it induces an isomorphism 
$A_\phi:A_{G'}\to A_G$ between the corresponding Lie algebroids. 
Given $G$ a Lie groupoid, its source-connected component $G^0\subset G$ is an open subgroupoid with the same objects and the same Lie algebroid, and therefore, $\phi:G'\to G$ is a Lie equivalence if and only if $\phi:G'^0\to G^0$ is so. In particular, the inclusion $G^0\to G$ is a Lie equivalence.
We will focus our attention on Lie equivalences $\phi$ between source-connected Lie groupoids. Without loss of generality, we will assume that $M'=M$ and that $\phi\bar_M$ is just the identity.
Lie equivalences relate the several integrations of a given Lie algebroid.



\begin{example}
Lie groups are the same as Lie groupoids with a single object. A Lie equivalence between Lie groups $\phi:G'\to G$ is a homomorphism inducing isomorphism on the Lie algebras. If $G$ and $G'$ are connected then $\phi$ is surjective, its kernel $K$ is discrete and lies in the center, and it follows that $\phi:G'\to G'/K\cong G$ is a topological covering (see e.g. \cite[I.11]{k}).
\end{example}

\begin{example}
If $M$ is connected, 
the projection $\phi:\pi_1(M)\to P(M)$ from its fundamental groupoid to its pair groupoid is a Lie equivalence, as both integrate $TM\then M$. There is a correspondence between (pointed) transitive groupoids and principal bundles (see eg. \cite[3.5.3]{dh}), under which the Lie equivalence $\phi$ is associated with the universal covering map $\tilde M\to M$. 
\end{example}

The next example shows the relevance of the source-connected hypothesis.

\begin{example}
Let $G\toto M$ be an \'etale Lie groupoid, namely one on which $G$ and $M$ have the same dimension. Then the Lie algebroid $A_G\then M$ is the zero vector bundle and the inclusion from the unit groupoid
$U(M)\to G$ is a Lie equivalence. This means, for instance, that if $G\action M$ is a discrete group acting on a manifold, then the dynamics is not seen at the infinitesimal level.
\end{example}

\medskip


Unlike the case of Lie groups, a Lie equivalence may not be a covering map at the level of arrows, even if we restrict ourselves to Hausdorff source-connected Lie groupoids, as we will see later in Example \ref{ex:deformation}.
In any case, there are strong ties between Lie equivalences and coverings maps. 
We organize them and other basic properties in the following proposition.

\begin{proposition}\label{prop:lie-equiv}
Let $\phi:G'\to G$ be a Lie equivalence between source-connected groupoids. Then (i) the map between the arrows $\phi_1:G'_1\to G_1$ is \'etale, namely a local diffeomorphism, (ii) the orbits of $G'$ and $G$ agree, and (iii)
the map between the isotropies $\phi_x:G'_x\to G_x$ and between the source-fibers $\phi\bar_{G'(-,x)}:G'(-,x)\to G(-,x)$ are covering maps.
\end{proposition}

\begin{proof}
Given $y\xfrom g x$ an arrow in $G$, the right multiplication $R_g:G(-,y)\to G(-,x)$ and the source map yield a natural short exact sequence relating the fiber of the Lie algebroid $A_y=\ker_{u(y)}ds=T_{u(y)}G(-,y)$, the tangent to the arrows and the tangent to the objects:
$$0 \to A_y\xto{d R_g}T_gG \xto{ds} T_xM\to 0$$
If $\phi$ is a Lie equivalence, then it yields isomorphisms on $A_y$ and $T_xM$, and by the five lemma, it should also induce another on $T_gG$. This proves the first claim.

Given $G\toto M$ a Lie groupoid and $x\in M$, the anchor map $\rho_x:A_x\to T_xM$ is natural, it has kernel the Lie algebra of the isotropy $G_x$ and cokernel the tangent space to the orbit $T_xO_x$.  Then, a Lie equivalence $\phi$ should preserve the kernel and cokernel of the anchor. Since in a source-connected groupoid the orbits are connected, (ii) follows, and it also follows that the morphisms between the isotropy groups are Lie equivalences.

For $x\in M$, we see $\phi\bar_{G'(-,x)}:G'(-,x)\to G(-,x)$ as a map over the orbit $O_x\subset M$ using $t'\bar_{G'(-,x)}$ and $t\bar_{G(-,x)}$, 
which are principal bundles with groups $G'_x$ and $G_x$ via right multiplication. Thus we can locally split $\phi\bar_{G'(-,x)}$ as $\phi_x\times\id:G'_x\times U\to G_x\times U$, from where both its image and its complement are opens. Since $G$ is source-connected, $\phi\bar_{G'(-,x)}$ is surjective, and so are the morphisms $\phi_x$.
Claim (iii) easily follows from here. 
\end{proof}

The previous proposition implies that source-simply-connected Lie groupoids are {\bf maximal integrations} among the source-connected ones. In fact, if $G$ is source-simply connected, the covering maps $\phi\bar_{G'(-,x)}:G'(-,x)\to G(-,x)$ are one-to-one for every $x$, and so does $\phi:G'\to G$.



%
%


Lie equivalences can be characterized by their kernels.
The {\bf kernel} $K\subset G'$ of a groupoid morphism $\phi:G'\to G$ is the subgroupoid of arrows that are mapped into an identity. 
It is wide, namely it has the same objects as $G'$, and it is normal, namely if \smash{$x \xfrom k x$} is in $K$ then \smash{$y\xfrom{g^{-1}kg}y$} is also in $K$ for every \smash{$x\xfrom g y$}. If $\phi$ is injective on objects then $K$ is intransitive, namely its source and target maps agree. The kernel may fail to be a smoothly embedded submanifold in general. We say that a subgroupoid $K\subset G$ is {\bf swind} if it is smoothly embedded, wide, intransitive, normal, and have discrete isotropy.

\begin{proposition}\label{prop:kernel}
Let $G'$ be a source-connected Lie groupoid, possibly non-Hausdorff. There is a 1-1 correspondence between Lie equivalences $\phi:G'\to G$ with $G$ source-connected and swind subgroupoids $K\subset G'$. Moreover, $G$ is Hausdorff if and only if $K$ is closed. 
\end{proposition}

\begin{proof}
Given a Lie equivalence $\phi:G'\to G$ with both $G',G$ source-connected, 
$\phi$ is a submersion between the arrows and between the isotropies, as shown in Proposition \ref{prop:lie-equiv}, so the kernel $K=\phi^{-1}(M)\subset G'$ is smooth, intransitive, and the isotropies $K_x\subset G'_x$ are discrete.
 
Conversely, given a swind subgroupoid $K\subset G'$, right multiplication defines a principal Lie groupoid action $G'\times_M K \to G'$,  $(g,k)\mapsto gk$, so we can apply Godement criterion (see \ref{prop:godement}, \ref{principal}).
It follows that the orbit space $G=G'/K$ is a manifold and the projection $G'\to G$ a submersion. Since $K$ is normal, $G$ inherits a groupoid structure over $M$, becoming a Lie groupoid. To see that the quotient map $G'\to G$ is a Lie equivalence, note that it yields a fiberwise epimorphism $A_{G'}\to A_G$ between the Lie algebroids and that both algebroids have the same rank.

Finally, since a Lie equivalence $\phi:G'\to G$ is a quotient map, $G$ is Hausdorff if and only if $M$ is closed, see Lemma \ref{lemma:units-are-closed}, and this is the case if and only if $K=\phi^{-1}(M)\subset G'$ is closed.
\end{proof}


The previous proposition can be found in \cite{gl}.
It can also be seen as an instance of a more general result, characterizing fibrations by their {\em kernel systems} (cf. \cite[\S 1.2.4]{mkbook}). A Lie equivalence $\phi:G'\to G$ is an example of a fibration, and its kernel system is simply the kernel $K\subset G'$. 

\begin{example}\label{ex:deformation}
Let $G'=\R^2\toto \R$ be the trivial line bundle over the real line, viewed as an intransitive source-connected Lie groupoid with source and target the (first) projection and multiplication the fiberwise addition. Then 
$K=\{(t,n/t):t\neq0,n\in\mathbb Z\}\subset G'$ is a closed swind subgroupoid, the projection $G'\to G=G'/K$ is a Lie equivalence between Hausdorff Lie groupoids by Prop. \ref{prop:kernel}, and it is not a covering map at the level of arrows. $G$ is the deformation space of $S^1$ into its Lie algebra, its isotropy groups are $G_t=S^1$ for $t\neq 0$ and $G_0=\R$. 
\end{example}

For the sake of completeness, we include the following lemma.

\begin{lemma}\label{lemma:units-are-closed}
A Lie groupoid $G\toto M$ is Hausdorff if and only if $M\subset G$ is closed.
\end{lemma}
\begin{proof}
If the sequence $g_n$ has two different limits $g,g'\in G$ then $g,g'$ must have the same source and target, for $M$ is Hausdorff, and then $g'^{-1}g=\lim (g_n)^{-1}g_n$ is in the closure of the units, so $M$ is not closed. Conversely, if $M$ is not closed then there is a sequence $u(x_n)\to g$ with $g$ not a unit, then $u(x_n)=usu(x_n)$ has two limits $g,u(s(g))$ and $G$ cannot be Hausdorff. 
\end{proof}


\section{An algebraic approach to holonomy}

\label{sec:holonomy}


We propose an algebraic approach to the monodromy and holonomy groupoids, which is equivalent to the one in the literature, provides an alternative insight into holonomy, and makes explicit some of their fundamental properties.

\medskip


Given $M$ a manifold and $F$ a foliation, by a {\bf foliated chart} $(U,\phi)$ we mean a chart
$\phi=(\phi_1,\phi_2):U\xto\sim\R^p\times\R^q$ that is a foliated diffeomorphism between $F\bar_U$ and the foliation by the second projection.
Given a foliated chart $(U,\phi)$, the {\bf local monodromy groupoid} $\mon(F\bar_U)$ is the Lie groupoid arising from the submersion $\phi_2:U\to\R^q$. Its objects are $U$, and it has one arrow $y\from x$ if $\phi_2(x)=\phi_2(y)$,  there is no isotropy and the orbits are the plaques.
An inclusion of foliated charts $U\subset V$ yields an inclusion $\mon(F\bar_U)\to \mon(F\bar_V)$. The {\bf monodromy groupoid} of $F$ can be defined as the colimit of the local monodromy groupoids and inclusions:
$$\mon(F) = \colim_{(U,\phi)} \mon(F\bar_U)$$

$\mon(F)$ is well-defined, at least set-theoretically, for the category of groupoids is cocomplete, namely every colimit exists \cite[p. 4]{gz}.
We will later show that this naturally inherits a smooth structure, but first, let us relate our approach with the one in the literature (eg. \cite{c,mmbook}).

\begin{lemma}\label{lemma:paths}
Set-theoretically, $\mon(F)$ is the disjoint union of the fundamental groupoids of the leaves, namely it has objects the points of $M$, and an arrow $y\xfrom g x$ in $\mon(F)$ identifies with the homotopy class of a path $y\overset\gamma\flowsfrom x$ within a leaf. 
\end{lemma}

\begin{proof}
Every groupoid has an underlying graph, consisting of its objects, arrows, source and target. The 3-steps construction of a groupoid colimit $\colim_{\alpha} G^\alpha$ goes as follows \cite[p. 4]{gz}: 
(i) compute the graph colimit $G^\infty$ levelwise, namely $G^\infty_0 =\colim_{\alpha} G^\alpha_0$, and $G^\infty_1 =\colim_{\alpha} G^\alpha_1$, and $s^\infty,t^\infty$ are the map induced by the $s^\alpha,t^\alpha$;
(ii) build the path category $P(G^\infty)$, with the same objects as $G^\infty$ and arrows the chains of arrows in $G^\infty$, and 
(iii) mod out $P(G^\infty)/\sim$ by all the relations spanned by the commutative triangles on each $G^\alpha$.

From this, it is rather clear that the objects of $\mon(F)$ are the points of $M$, and that we can regard an arrow $y\xfrom g x$ in $\mon(F)$ as the class of a discrete path $(g_k,\dots,g_1)$ where $g_i=(y_i\from x_i)_{U_i}\in \mon(F\bar_{U_i})$, $x_i=y_{i-1}$, $x_1=x$ and $y_k=y$, under the equivalence relation generated by:
\begin{itemize}
 \item[(i)] replacing some $g_i$ by $\iota(g_i)$ if $\iota:\mon(F\bar_{U})\to \mon(F\bar_{V})$ is a chart inclusion;
 \item[(ii)] replacing $g_i,g_{i-1}$ by the product $g_ig_{i-1}$ if they belong to the same chart; and
 \item[(iii)] insert $h=id_{x_i}\in\mon(F\bar_{U_i})$ between $g_i$ and $g_{i-1}$. 
\end{itemize}
To each arrow $(g_k,\dots,g_1)$ we can associate the juxtaposition of the segment paths within each chart, hence defining a groupoid map 
$\mon(F)\to\coprod_L \pi_1(L)$. 
The proof that this is indeed a groupoid isomorphism can be done leafwise, and it is a basepoint-free version of Van Kampen theorem, similar to that in \cite[1.7]{may}, subdividing continuous paths and homotopies into small enough pieces, each of them included in some foliated chart.
\end{proof}


Given $(U,\phi)$, $(U',\phi')$ foliated charts and given $x\in U\cap U'$, the {\bf transverse transition map} at $x$ is the (germ of a) diffeomorphism
$\gamma_{U'U}^x:(\R^q,\phi_2(x))\to(\R^q,\phi'_2(x))$ given by $\gamma_{U'U}^x(y)=(\phi'\phi)_2^{-1}(\phi_1(x),y)$.
More generally, given $y\xfrom g x$ in $\mon(F)$, and given $(U,\phi)$, $(U',\phi')$ foliated charts around $x$ and $y$, the {\bf holonomy} $\gamma^g_{U'U}:(\R^q,\phi_2(x))\to(\R^q,\phi'_2(y))$ of $g$ with respect to $U,U'$ is defined by representing $g$ as a discrete path
$(g_k,\dots,g_1)$, $g_i\in\mon(F\bar_{U_i})$, $U=U_1$, $U'=U_k$, as the composition of the transverse transition maps $\gamma^{t(g_i)}_{U_{i+1}U_{i}}$. This is well-defined for the above composition is invariant under the three elementary moves described in Lemma \ref{lemma:paths}. 


\begin{proposition}\label{prop:mon-hol}
The monodromy groupoid $\mon(F)$ is naturally a Lie groupoid, possibly non-Hausdorff. The inclusions $\mon(F\bar_U)\to\mon(F)$ are smooth, and the colimit is valid within the category of Lie groupoids.
\end{proposition}

\begin{proof}
Given $y\xfrom g x$ in $\mon(F)$, we show now how to construct a foliated chart around $g$. Regard $g$ as the class of a discrete path $(g_k,\dots,g_1)$ with $g_i=(y_i\from x_i)\in\mon(F\bar_{U_i})$, and realize  the germ $\gamma^g_{U_kU_1}$ as a diffeomorphism $B\to\gamma^g_{U_kU_1}(B)$ with $(\phi_1)_2(x)\in B\subset\R^q$ an open ball. Then we can set $(\tilde U,\tilde\phi)$ a chart for $\mon(F)$ around $g$ by
$$\tilde U=\{(g'_k,\dots,g'_1): x'_i\xfrom{g'_i}x'_{i-1}\in \mon(F\bar_{U_i}),\phi_2(x'_1)\in B\}$$
and by $\tilde\phi:\tilde U\to (\R^p\times\gamma^g_{U_kU_1}(B))\times_{B}(\R^p\times B)$, 
$g'\mapsto(\phi_k(x'_k),\phi_1(x'_0))$, where the fiber product over $\pi_2(\gamma^g_{U_kU_1})^{-1}$ and $\pi_2$ is an Euclidean open of dimension $2p+q$. Given two such charts and fixing an arrow $g$ on that intersection, it is straightforward to check that the three elementary moves lead to a smooth transition map, well-defined around $g$, from where the result easily follows.

It easily follows that the inclusions $\mon(F\bar_U)\to\mon(F)$ are open embeddings covering a neighborhood of the identities. The colimit is valid within the category of Lie groupoids because a groupoid map $\mon(F)$ is smooth if and only if it is so in such a neighborhood.
\end{proof}


If $x\xfrom g x$ is an arrow in $\mon(F)$ and $(U,\phi)$ is a foliated chart around $x$, then we can build a chart $(\tilde U,\tilde\phi)$ around $g$ as above, by using $(U,\phi)$ as the initial and final foliated chart, and then
$$\tilde\phi(\tilde U\cap I(\mon(F)))=
\{(x,\gamma^g_{UU}(y),x,y)\}\cap\{(x,y,x,y)\}\subset (\R^p\times\gamma^g_{U_kU_1}(B))\times_{B}(\R^p\times B),$$
where $I(\mon(F))$ denotes the isotropy of $\mon(F)$. It follows that $s:I(\mon(F))\to M$ is always locally injective, and it is locally surjective at $g$ if and only if $\gamma^g_{UU}$ is trivial. Based on this, we say that $g$ has {\bf trivial holonomy} if $s:I(\mon(F))\to M$ is locally bijective at $g$. The arrows with trivial holonomy $K^h\subset\hol(F)$ define a swind subgroupoid, so the quotient of $\mon(F)$ by $K^h$ is a well-defined {\bf holonomy groupoid} $\hol(F)\toto M$, and the projection $\mon(F)\to\hol(F)$ is a Lie equivalence, see Proposition \ref{prop:kernel}. While $\mon(F)$ is source-simply connected, $\hol(F)$ is just source-connected and its isotropy groups are the {\bf holonomy groups} $\hol_x(F)$. Two paths $g,g'$ with the same initial and final points have the same holonomy if they induce the same diffeomorphism on small transversals.

\begin{proposition}[cf. {\cite[Prop. 1]{cm}}]\label{prop:min-max}
If $G$ is a source-connected Lie groupoid integrating $F$, then 
the canonical projection $\mon(F)\to\hol(F)$ factors through $G$.
\end{proposition}

\begin{proof}
Given $(U,\phi)$ a foliated chart, the orbits of $(G_U)^0$ are the horizontal plaques $\phi^{-1}(\R^p\times y)$, the map $(t,s): (G_U)^0\to \mon(F\bar_U)$ is a Lie equivalence, and since the source-fibers of $\mon(F\bar_U)$ are simply connected, $\psi$ is an isomorphism by \ref{prop:lie-equiv}. The inclusions 
$\mon(F\bar_U)\cong(G_U)^\circ\to G$ induce a morphism from the colimit $\phi:\mon(F)\to G$ preserving the underlying foliation, hence being a  Lie equivalence. 
Since $G$ is source-connected, it follows from \ref{prop:kernel} that $\phi$ is surjective and $G=\mon(F)/\ker\phi$. To get the morphism $G\to\hol(F)=\mon(F)/K^h$ we need to show that $\ker\phi\subset K^h$, namely that the arrows in $\ker(\phi)$ have trivial holonomy, which by definition means that $s:{I(\mon(F))}\to M$ is locally surjective at every $g\in \ker(\phi)$. This follows from the inclusion $\ker(\phi)\subset I(\mon(F))$ and the fact that $s:\ker(\phi)\to M$ is already a surjective submersion.
\end{proof}


As we have seen, the monodromy and the holonomy groupoid are always Lie groupoids, though the manifold of arrows may be non-Hausdorff. Let us illustrate with some simple examples.

\begin{example}\label{ex:monodromy}
Let $F$ be the foliation on $M=\R^3\setminus 0$ by horizontal planes. $F$ is simple, its leaves are the fibers of $(x,y,z)\mapsto z$, there is no holonomy, and $\hol(F)$ is a submersion groupoid, hence Hausdorff.
On the other hand, $\mon(F)$ is non-Hausdorff, for $u(1,0,1/n)\in\mon(F)$ converges to any element of the isotropy at $(1,0,0)$, that is isomorphic to $\mathbb Z$.
\end{example}

\begin{example}\label{ex:holonomy}
Write $f:\R\to\R$ for the smooth map given by $f(t)=e^{-1/t}$ if $t>0$ and $f(t)=0$ otherwise. 
Let $M=\R\times S^1$ be the cylinder with coordinates $t,r$ and $F$ the 1-dimensional foliation spanned by the vector field
$X(t,r)=\frac{\partial}{\partial r} + f(t)\frac{\partial}{\partial t}$. Then $\mon(F)$ is Hausdorff, it is the action groupoid $\R\ltimes M$ given by the flow of $X$, while $\hol(M)$ is not Hausdorff at the origin.
\end{example}

We close with an example of a foliation that does not admit a Hausdorff integration.

\begin{example}\label{ex:hol=mon}
Let $M=\R^3\setminus\{(0,0,z):z\geq 0\}$, and let $F$ be the foliation given by the 1-form $\frac{f(z)y}{x^2+y^2}dx - \frac{f(z)x}{x^2+y^2}dy + dz$. The leaves of $F$ on $z<0$ are the horizontal planes, and on $z>0$ are spirals spanned by the vector fields $x\frac{\partial}{\partial x}+y\frac{\partial}{\partial y}$ and $-y\frac{\partial}{\partial x}+x\frac{\partial}{\partial y}+f(z)\frac{\partial}{\partial z}$.  Then $\mon(F)$ is non-Hausdorff, as the non-trivial loops at $z=0$ are in the closure of the units. And since these are all the non-trivial loops and they have non-trivial holonomy, $K^h=M$ and $\hol(F)=\mon(F)$. It follows from Proposition \ref{prop:min-max} that $F$ does not admit any Hausdorff integration.
\end{example}


\section{A maximal Hausdorff integration}


We present here our first main contribution, which is a Hausdorff version of Lie 1. We start reviewing the classic version, establish the new result, characterize the associated kernel via vanishing cycles, and illustrate with several examples.

\medskip


Given $G$ a Lie groupoid, Lie 1 ensures the existence of a maximal Lie equivalence $\tilde G\to G$ \cite{mmbook}. We restate it here by elucidating the universal property that $\tilde G$ satisfies. And we give a proof based on the monodromy and holonomy groupoids.

\begin{proposition}[Lie 1]
\label{prop:lie1}
Given $G$ a Lie groupoid, there exists a Lie groupoid $\tilde G$ and a universal Lie equivalence $\tilde\phi:\tilde G\to G$, in the sense that for any other Lie equivalence $\phi':G'\to G$ there exists a unique factorization $\tilde\phi=\phi'\phi$. 
$$\xymatrix{
\tilde G \ar[r]^{\exists ! \phi} \ar[d]_{\tilde\phi}  & G' \ar[dl]^{\forall \phi'}  \\ G
}$$
\end{proposition}


\begin{proof}
Without loss of generality, we can assume $G$ source-connected.
Let $F^s$ be the foliation on $G$ by the source-fibers, which is invariant under the free $G$-action by right multiplication. Even though $G$ may be non-Hausdorff, it still makes sense to build the groupoids $\mon(F^s)$ and $\hol(F^s)$, as done in the previous section. Right multiplication induces $G$-actions on
$\mon(F^s)$ and $\hol(F^s)$ for which the source, target and  multiplicaction are $G$-equivariant. It follows from \ref{principal} that the actions are principal.
Then by Godement criterion \ref{prop:godement} the quotients are well-defined Lie groupoids with objects $G/G=M$. Since $F^s$ has no holonomy, $\hol(F^s)$ is the submersion groupoid induced by $s:G\to M$ and $\hol(F^s)/G = G$. Let $\tilde G$ be the quotient $\mon(F^s)/G$. The projection $\mon(F^s)\to\hol(F^s)$ induces a Lie equivalence $\tilde\phi:
\tilde G \to G$.
To see that it is universal, let $\phi':G'\to G$ be a Lie equivalence, and consider the Lie groupoid theoretic fibered product $G'\times_G \hol(F^s)$  \cite[Appendix]{bcdh}. Since $\phi'$ is a Lie equivalence, the same holds for the base-change morphism $G'\times_G \hol(F^s)\to\hol(F^s)$, so $G'\times_G \hol(F^s)$ integrates $F^s$. Then Proposition \ref{prop:min-max} gives a uniquely defined Lie equivalence 
$\mon(F^s)\to G'\times_G \hol(F^s)$, which modding out by the $G$-action gives the desired map
$\phi:\tilde G\to G'$.
\end{proof}


Note that since $\tilde\phi$ and $\phi'$ are Lie equivalences, the same holds for $\phi$. The source-fibers of $\tilde G$ identify with those of $\mon(F^s)$, from where it is clear that $\tilde G$ is source-simply connected.
Our proof exploits that every Lie groupoid is the quotient of a holonomy groupoid. The infinitesimal analog to this statement will play a key role in the next section.



\begin{theorem}[Hausdorff Lie 1]
\label{thm:HLie1} 
Given $G$ a Hausdorff Lie groupoid, there exists a Hausdorff Lie groupoid $\hat G$ and a universal Lie equivalence $\hat\phi:\hat G\to G$, in the sense that for any other Lie equivalence $\phi':G'\to G$ with $G'$ Hausdorff, there exists a unique factorization $\hat\phi=\phi'\phi$. 
$$\xymatrix{
\hat G \ar[r]^{\exists ! \phi} \ar[d]_{\hat\phi}  & G' \ar[dl]^{\forall \phi'}  \\ G
}$$
\end{theorem}

\begin{proof}
We can assume $G$ source-connected. Let $\tilde\phi:\tilde G\to G$ be the universal Lie equivalence given by Lie 1, and let $\tilde K\subset \tilde G$ be its kernel, which is closed by Proposition \ref{prop:kernel}.

Given a Lie equivalence $\phi':G'\to G$ from a Hausdorff groupoid $G'$ we denote by $K'\subset \tilde G$ the kernel of the factorization $\phi:\tilde G\to G'$, which is closed and is included in $\tilde K$. Then we can identify $G'=\tilde G/K'$, see Proposition \ref{prop:kernel}.
In particular, if $\hat\phi:\hat G\to G$ is a universal Lie equivalence, it follows from $\hat\phi=\phi'\phi$ that the kernel $\hat K$ of $\tilde G\to \hat G$ is included in $K'$, for every $K'$. Thus, to construct $\hat G$, we need a minimal closed swind subgroupoid $\hat K$ inside $\tilde K$. 

If $M\subset K\subset \tilde K$ is a subgroupoid then $K$ is automatically wide, intransitive and with discrete isotropy. Define $\hat K$ as the intersection of all the closed swind subgroupoid $K$ of $\tilde K$. The intersection is non-trivial, for at least $K=\tilde K$ is closed swind. Clearly $\hat K$ is closed and normal. To show that $\hat K$ is swind, we only need to prove that $\hat K$ is smooth, or equivalently open, as $M$ and $\tilde K$ are manifolds of the same dimension.
Given $K\subset \tilde K$ one of the groupoids we are intersecting, if $g\in\hat K$, and if $g\in U\subset\tilde K$ is a connected open neighborhood, then $U\subset K$ for every $K$, as both $K$ and $\tilde K\setminus K$ are open. It follows that $U\subset \hat K$, so $\hat K$ is open, as claimed.
 
Finally, since $\hat K$ is closed swind, the quotient $\hat G=\tilde G /\hat K$ is a Hausdorff Lie groupoid, again by Proposition \ref{prop:kernel}, and the map $\tilde G\to \hat G$ is a Lie equivalence, as well as $\hat G\to G$. By construction, the latter is universal among the Lie equivalences from a Hausdorff groupoid.
\end{proof}


$\hat K\subset\tilde K$ is both open and closed, and therefore it has to be a union of connected components. In particular, $\hat K$ must contain every component of $\tilde K$ intersecting $M$. 
We can think of arrows in $\tilde K$ as $G$-classes of (homotopy types of) loops within a leaf of the foliation $F^s$. 
A non-trivial loop $\alpha_0$ in a leaf $L_ 0$ is a 
{\bf vanishing cycle} if it can be extended to a continuous family $\alpha_t$, $0\leq t\leq 1$, such that $\alpha_t$ is a trivial loop on some leaf $L_t$ for all $t>0$. 
Every vanishing cycle of $F^s$ must belong to $\hat K$, and $F^s$ has no vanishing cycles if and only if $M$ is closed in $\tilde G$.

\begin{corollary}\label{cor:ssch}
Given $G$ a Hausdorff groupoid, the maximal integration $\tilde G$ is Hausdorff, namely $\tilde G=\hat G$, if and only if the foliation $F^s$ on $G$ has no vanishing cycles.
\end{corollary}

The next examples show that $\hat G$ can be either $\tilde G$ or $G$ or something in between.

\begin{example}
Given $G$ a Hausdorff groupoid, if the source map $s:G\to M$ is trivial, as in the case of groupoids arising from Lie group actions $K\action M$, and more generally, if $s$ is locally trivial, as in the case of strict linearizable groupoids (cf. \cite{dhf1}), then there are no vanishing cycles on $F^s$, and therefore, the universal groupoid $\tilde G$ is Hausdorff and equal to $\hat G$.
\end{example}

\begin{example}
Let $G$ be the holonomy groupoid of the foliation $F$ by horizontal planes on $\R^3\setminus 0$. In Example \ref{ex:monodromy} we see that
$G$ is Hausdorff and that $\tilde G=\mon(F)$ is not. The only non-identity arrows in the kernel $\tilde K$ are the non-trivial loops in the leaf $z=0$, and they are all vanishing cycles, so $\hat K=\tilde K$ and the universal Hausdorff groupoid $\hat G$ is in this case equal to $G$.
\end{example}

\begin{example}
We can modify the previous example by considering the foliation $F$ by horizontal planes on $\R^3\setminus(\{0\}\cup L)$, where $L$ is a vertical line other than the $z$-axis, and setting $G=\hol(F)$. Then $G$ is Hausdorff, $\tilde G=\mon(F)$ is non-Hausdorff, and $\hat G$ does not agree with $G$ nor $\tilde G$. The kernel $\hat K$ has the non-trivial loops coming from the missing point, but it does not contain the non-trivial loops corresponding to the missing line.
\end{example}

\section{Integrating Lie subalgebroids}

We generalize now the results on the integration of Lie subalgebroids from \cite{mm-s} to the non-wide case, namely when the subalgebroid is defined over a proper submanifold. This extension will be used later to give a nice conceptual proof of Lie 2 that is adaptable to the Hausdorff case.

\medskip 
 

Given $G$ a Lie groupoid, the foliation $F^s$ on $G$ by s-fibers can be seen as the pullback vector bundle of its algebroid $A_G$ through the target map, or alternatively, we can see $A_G$ as the quotient of $F^s$ under the action of $G\action G\xto t M$ by right multiplication, which keeps this foliation invariant. This way every Lie algebroid is the quotient of a foliation. 
$$\xymatrix{
F^s \ar[r] \ar@{=>}[d] & A_G \ar@{=>}[d] \\ G_1 \ar[r]^t & M
}$$

\begin{proposition}\label{prop:subalgbd}
Given $G\toto M$ a Lie groupoid and $S\subset M$, there is a 1-1 correspondence between Lie subalgebroids $B\then S$ of $A_G\then M$ and $G$-invariant foliations $F$ on $t^{-1}(S)\subset G$.
\end{proposition}
 
\begin{proof}
Given $B\then S$ a Lie subalgebroid, we can compute the Lie-algebroid-theoretic fibered product between the projection $F^s\to A_G$ and the inclusion $B\to A_G$ (cf. \cite{bcdh}), which turns out to be a subalgebroid $t^*B=F$ of $F^s$. Since a subalgebroid of a foliation must be a foliation, $F$ is a foliation itself, and since $F^s$ is $G$-invariant for the right multiplication, so does $F$.
$$\xymatrix@R=10pt@C=10pt{
t^*B=F \ar[rr]\ar@{=>}[dd] \ar[rd] & & B=F/G \ar@{=>}'[d][dd] \ar[rd] & \\
 & F^s \ar[rr] \ar@{=>}[dd]  & & A_G \ar@{=>}[dd] \\
t^{-1}(S) \ar'[r][rr] \ar[rd] & & S \ar[rd] & \\ 
& G \ar[rr] & & M}
$$

Conversely, given $F$ a $G$-invariant foliation on $t^{-1}(S)$, 
since the maps $F\to t^{-1}(S)\to G$ are $G$-equivariant, the actions $F\curvearrowleft G$ and $t^{-1}(S)\curvearrowleft G$  are both principal by \ref{principal}. Then by Godement criterion \ref{prop:godement} we get a quotient vector bundle $B=F/G$ over $S=t^{-1}(S)/G$, which inherits the structure of a Lie algebroid, by identifying its sections with right-invariant vector fields. The inclusion $F\subset F^s$ yields an algebroid inclusion in the quotients $B\subset A_G$.
It is straightforward to check that these two constructions are mutually inverse.
\end{proof}


Let $G\toto M$ be a Lie groupoid, $B\then S$ a non-necessarily wide subalgebroid of $A_G\then M$, and $F^B$ the corresponding $G$-invariant foliation on $t^{-1}(S)$.
The principal action $G\action t^{-1}(S)$ induces principal actions $G\action\mon(F^B)$ and $G\action\hol(F^B)$ by \ref{principal}, and by Godement criterion \ref{prop:godement} there  are Lie groupoid quotients, which in the notations of \cite{mm-s}, are $H_{max}$ and $H_{min}$.
$$H_{max}=\mon(F^B)/G  \qquad H_{min}=\hol(F^B)/G$$


Next we extend Proposition \ref{prop:min-max} and
generalize some results from \cite{mm-s} to the non-wide case.  

\begin{proposition}\label{prop:min-max-2}
The Lie groupoids $H_{max}$ and $H_{min}$ are integrations of $B\then S$. The canonical map $H_{max}\to H_{min}$ is a Lie equivalence. The inclusion $B\subset A_G$ integrates to immersive morphisms $H_{max}\to \tilde G$ and $H_{min}\to G$ fitting in the following commutative diagram:
$$\xymatrix{
H_{max}\ar[r]^{} \ar[d]  & \tilde G \ar[d] \\
H_{min}\ar[r]^{} & G
}$$
\end{proposition}

\begin{proof}
Since $\mon(F^B)$ and $\hol(F^B)$ are integrations of $F^B$, the quotients $H_{max}$ and $H_{min}$ are integrations of $F^B/G=B$, and since $\mon(F^B)\to\hol(F^B)$ is the identity at the infinitesimal level, the same holds for quotient morphism $H_{max}\to H_{min}$.

The inclusion $F^B\subset F^s$ gives rise to a commutative diagram of monodromy and holonomy groupoids. Even though the holonomy groupoid is not functorial, it does yield a map in this particular situation, for the foliation $F^s$ has no holonomy.

$$\xymatrix{
\mon(F^B) \ar[r]^{} \ar[d]  & \mon(F^s) \ar[d] \\
\hol(F^B) \ar[r]^{} & \hol(F^s)
}$$
Modding out by the $G$-action we get the desired commutative square. That $H_{max}\to \tilde G$ and $H_{min}\to G$ are immersive follows from the natural sequence
$0 \to A_y\xto{d R_g}T_gG \xto{ds} T_xM\to 0$.
\end{proof}

The previous proposition shows that a subalgebroid of an integrable algebroid is integrable, for instance, by $H_{min}$. But the canonical map $H_{min}\to G$ induced by the inclusion need not be injective. As an example, consider the pair groupoid $G=P(M)$ and a foliation $B=F\subset TM=A_G$ with non-trivial holonomy.

\begin{corollary}
\label{cor:trivial-holonomy}
$H_{min}\rightarrow G$ is injective if and only if the foliation $F^B$ has trivial holonomy.
\end{corollary}

\begin{proof}
In the diagram below, one of the horizontal morphisms is injective if and only if the other is so, as the vertical maps are principal $G$-bundles. 
$$\xymatrix{
\hol(F^B) \ar[r]^{} \ar[d]  & \hol(F^s) \ar[d] \\
H_{min}\cong \hol(F^B)/G \ar[r]^{} & G\cong \hol(F^s)/G
}$$
And regarding $(\hol(F^B)\then t^{-1}(N))\to(\hol(F^s)\then G)$, since it is injective at the level of objects, it is going to be injective at the level of arrows if and only if it is so in the isotropies. Since $F^s$ has no isotropy, this holds if and only if $F^B$ has no holonomy, and the result follows.
\end{proof}



\section{Hausdorff versions for Lie's second Theorem}


Using the results from the previous section, we present here a simple proof of classic Lie 2. Then we derive a Hausdorff version, which must include a subtle hypothesis, as we show in an example. We work out a second Hausdorff version valid for foliation groupoids.

\medskip



\begin{proposition}[Lie 2]\label{prop:lie2}
Let $G$ and $H$ be Lie groupoids and $\varphi:A_G\rightarrow A_H$ a Lie algebroid morphism. If $G=\tilde G$ then $\varphi$ integrates to a groupoid morphism
$\phi:G\to H$, which is unique.
\end{proposition}

The idea of integrating the graph is already present in the original reference \cite{mx}, see also \cite{mmbook}. Our results about integrating non-wide subalgebroids help us to concretize this idea.

\begin{proof}
Let $(B\then S)\subset (A_G\times A_H\then M\times N)$ be the graph of $\varphi$, seen as a non-wide Lie subalgebroid of the cartesian product. Formally, we can construct $B\then S$ as the fibered product between the identity of $A_G$ and $\varphi$, as in \cite[Appendix]{bcdh}.
Since $G\times H$ integrates $A_G\times A_H$, Proposition \ref{prop:min-max-2}  yields a Lie groupoid $H_{min}$ integrating $B$ and a morphism $\alpha$ as follows: $$\xymatrix{
H_{min} \ar[r]^{\alpha} \ar^{\pi}[d]  & G\times H  \ar[ld]^{pr_1} \ar[dr]^{pr_2} & \\
\ar@/^/@{-->}[u]^{\beta} G  &  & H
}$$
The composition $pr_1\alpha:H_{min}\to G$ is a Lie equivalence, and since $G=\tilde G$, by Lie 1, there must exists a section $\beta$ for it. We can finally integrate $\varphi$ to the composition $\phi=pr_2\alpha\beta$.
 \end{proof}

 
Given $G, H$ Hausdorff Lie groupoids, and $\varphi: A_G\to A_H$ a Lie algebroid map, the integration $\phi:\tilde G\to H$ may not descend to the maximal Hausdorff quotient $\hat G$, as the next example shows.


\begin{example}\label{ex:counter}
Let $M=\R^3\setminus\{(0,0,z):z\geq 0\}$, and let $F$ be the simple foliation given by the third projection.
Let $G=\hol(F)$, which is a submersion groupoid, hence Hausdorff.
The kernel $\tilde K$ of the projection $\mon(F)\to\hol(F)$ is connected, for every loop $\gamma$ in $F$ can be deformed into a trivial one just by decreasing $z$. It follows that $\hat K=\tilde K$ and $\hat G=G$.
Let  $\phi:(\mon(F)\toto M)\to (\R\toto\ast)$ be the morphism given by
$\phi([\gamma])=f(z)\omega(\gamma)$, where 
$f:\R\to\R$ is the smooth map in Example \ref{ex:holonomy} and
$\omega(\gamma)=\frac{1}{2\pi}\int_\gamma \frac{xdy-ydx}{x^2+y^2}$ is the winding number of $\gamma$ at $0$.
$\ker\phi$ consists of the loops at $z=0$, so it does not include $\tilde K$. Thus, even though $\R\toto \ast$ is Hausdorff, $\phi$ does not descend to a morphism $\hat G\to\R$.
\end{example}



Let us present now our first Hausdorff version for Lie 2. Given $G,H$ Lie groupoids and $\varphi:A_G\to A_H$ a Lie algebroid morphism, denote by $B\then S$ the graph of $\varphi$, and by $F^\varphi$ to the $(G\times H)$-invariant foliation on $(t\times t)^{-1}(S)\subset G\times H$, given in Proposition \ref{prop:subalgbd}. Compare the holonomy hypothesis with that of Corollary 2.5 in \cite{mm-s}. 

\begin{theorem}[Hausdorff Lie 2, v1]
\label{thm:HLie2v1}
Let $G$ and $H$ be Hausdorff Lie groupoids and $\varphi:A_G\to A_H$ a Lie algebroid morphism. If $G=\hat G$ and the foliation $F^\varphi$ has trivial holonomy, then $\varphi$ integrates to a groupoid morphism $\phi:G\to H$, which is unique. 
\end{theorem}

\begin{proof}
As in Proposition \ref{prop:lie2}, the inclusion of the graph $S\subset A_G\times A_H$ of $\varphi$ integrates to a Lie groupoid immersion $\alpha:H_{min}\to G\times H$. Since $F^\varphi$ has no holonomy, the morphism $\alpha$ is injective, as seen in Corollary \ref{cor:trivial-holonomy}, and since $G\times H$ is Hausdorff we conclude that $H_{min}$ is Hausdorff as well.
The composition $pr_1\alpha:H_{min}\to G$ is a Lie equivalence, and since $G=\hat G$, by Theorem \ref{thm:HLie1}, there must exist $\beta:G\to H_{min}$ a section for $pr_1\alpha$. The composition $\phi=pr_2\alpha\beta$ is the desired integration of $\varphi$.
\end{proof}

Let us now focus on the case of {\bf foliation groupoids}, namely those on which the isotropy groups are discrete \cite{cm}. Foliation groupoids can be characterized as those whose algebroid has an injective anchor map, or equivalently, those integrating a foliation. They can also be characterized as those groupoids which are Morita equivalent to an \'etale one.

\begin{lemma}\label{lemma:foliation-units}
If $H\toto N$ is a foliation groupoid, then the units $N$ form an open set within the isotropy $I(H)=\bigcup_{x\in N}H_x$ of $H$.
\end{lemma}

\begin{proof}
The characteristic foliation $F$ of $H$ is a regular foliation on $N$. Let $x\in N$ and let $(U,\phi)$ be a foliated chart around it. Then $U=V\times W$ with $V\cong\R^p$, $W\cong\R^q$ and $\dim(M)=p+q$. We have seen in the proof of Proposition \ref{prop:min-max} that $H^\circ_U$ is the local monodromy groupoid $\mon(F\bar_U)$, and this is isomorphic to the product $P(V)\times U(W)$ of the pair groupoid of $V$ and the unit groupoid of $W$.
Since $H^\circ_U\subset H_U\subset H$ are open inclusions, we conclude that $I(H)\cap H^\circ_U=I(H^\circ_U)=U\subset N$ is an open neighborhood of $x$ within $I(H)$.
\end{proof}

Our second Hausdorff version for Lie 2 disregards the holonomy hypothesis when assuming that $H$ is a foliation groupoid.

\begin{theorem}[Hausdorff Lie 2, v2]
\label{thm:HLie2v2}
Let $G$ and $H$ be Hausdorff Lie groupoids and $\varphi:A_G\to A_H$ a Lie algebroid morphism. If $G=\hat G$ and $H$ is a foliation groupoid then $\varphi$ integrates to a groupoid morphism $\phi:G\to H$, which is unique. 
\end{theorem}

\begin{proof}
We know that $\varphi$ can be integrated to a morphism $\tilde{\phi}:\tilde G\to H$ by Lie 1. This descends to a morphism $\hat G\to H$ if and only if the kernel $\hat K$ of the projection $\tilde G\to \hat G $ is included in $\ker\tilde\phi$.

We know that $N$ is open in $I(H)$ by Lemma \ref{lemma:foliation-units}, and it is also closed by Lemma \ref{lemma:units-are-closed}. 
It follows that 
$\tilde\phi^{-1}(N)\cap\tilde K$ is a closed swind subgroupoid, and by the proof of Theorem \ref{thm:HLie1}, it contains $\hat K$, or in other words, $\hat K$ is included in the kernel of $\tilde\phi$, which completes the proof.
\end{proof}


Given a Lie algebroid $A$ and a Hausdorff integration $G$, unlike the case of $\tilde G$, the maximal Hausdorff integration $\hat G$ from Theorem \ref{thm:HLie1} strongly depends on $G$. More precisely, if $H$ is another Hausdorff integration of $A$, then $\hat H$ and $\hat G$ may a priori be different. In light of Proposition \ref{prop:kernel}, this is because the intersection of the two closed swind subgroupoids $\hat K_G,\hat K_H\subset \tilde G$ may fail to be smooth. Nevertheless, when $A$ is a foliation, the situation is simpler, as described below.

\begin{corollary}\label{cor:mhfc}
If $F$ is a foliation over $M$ admitting a Hausdorff integration, then there is a maximal Hausdorff integration $\hat G_F$ that covers any other Hausdorff integration of $F$.
\end{corollary}

\begin{proof}
Let $G_1$ and $G_2$ be Hausdorff Lie groupoids integrating $F$. Consider $\hat{G_1}$ and $\hat{G_2}$ the Hausdorff covering groupoids associated to $G_1$ and $G_2$, as constructed in Theorem \ref{thm:HLie1}. Since $G_1$ and $G_2$ are foliation groupoids, we can apply Theorem \ref{thm:HLie2v2} to integrate the identity of $F$ to Lie groupoid morphisms $\hat{G_1}\to\hat{G_2}$ and $\hat{G_2}\to\hat{G_1}$, 
which should be mutually inverse by the uniqueness of the integration.
It follows that $\hat{G_1}$ and $\hat{G_2}$ are isomorphic, hence defining $\hat G_F$.
\end{proof}


\section{Application to symplectic geometry}


We now apply our previous results to the integration of Poisson manifolds.
We show that if the Lie algebroid induced by a Poisson manifold has a Hausdorff integration, it also has a Hausdorff symplectic one. To do that, we follow the approach of \cite{bco}, where the canonical symplectic form on the cotangent of a Poisson manifold is regarded as a Lie algebroid morphism. We assume some familiarity with VB-groupoids and VB-algebroids, introduced and studied in \cite{pradines,dla1}. We follow the conventions and notations adopted in \cite{bcdh}.

\medskip


A Poisson manifold $(M, \pi)$ gives rise to an {\bf induced Lie algebroid}
$A\then M$ with $A=T^*M$ and bracket and anchor are given by $[df,dg]=d\{f,g\}$ and $\rho(df)=X_f$. It follows from the main theorem in \cite{bco} that the canonical symplectic form $\omega_{can}$ is compatible with the algebroid structure in $A=T^*M$, in the sense that the induced map $\omega_{can}^b:TA\to T^*A$ is VB-algebroid isomorphism between the tangent and cotangent structures $TA\then TM$ and $T^*A\then A^*$. 
It turns out that any Lie algebroid $A$ with a compatible symplectic structure $\omega\in\Omega^2(A)$ turns out to be a Poisson manifold, for the isomorphism $\omega^b$
$$\xymatrix{
TA \ar[r]^{\omega^b} \ar@{=>}[d]  & T^*A \ar@{=>}[d] \\
TM \ar[r] & A^*
}$$
allow us to identify $A\cong T^*M$, $\omega$ with $\omega_{can}$, and the core-anchor map of $TA\then TM$ with a Poisson bivector $\pi^\#:T^*M\to   TM$. Further details can be found in \cite{bco}.

Symplectic groupoids arise as the global counterpart of Poisson manifolds. A {\bf symplectic groupoid} $(G\toto M,\omega)$ is a Lie groupoid with a compatible symplectic form $\omega$ on $G$. 
The compatibility can be described in the following equivalent ways: (i) $\omega^b:(TG\toto M)\to (T^*G\toto A^*)$ is a VB-groupoid morphism, (ii) $\omega$ is multiplicative, or (iii) the graph of the groupoid multiplication is Lagrangian (see \cite{cdw}, \cite{mkbook}). A symplectic groupoid $G$ induces a Poisson structure on the units $M$ such that the source map $s: G \to M$ is Poisson, hence a complete symplectic realization. \cite{cf2} proves that a Poisson manifold admits a complete symplectic realization if and only if the associated algebroid is integrable.



Given $(M,\pi)$ a Poisson manifold, $A\then M$ the induced Lie algebroid, and $G\toto M$ an integration of it, the canonical form $\omega_{can}$ on $A$ may not be integrable to a symplectic form $\omega$ on $G$, so $G\toto M$ may not be a symplectic groupoid. We illustrate it with a simple example.

\begin{example}
Let $M=S^3$ and $\pi=0$. The induced Lie algebroid is $T^*M\then M$, with zero bracket and anchor map. An integration of $A$ is the cotangent bundle $G=T^*M\toto M$, with fiberwise addition as multiplication. Any other integration is obtained from $G$ by modding out by a wide discrete group bundle $K$ that is Lagrangian. When $K$ has rank 3, the quotient $G/K\cong S^3\times T^3$ is compact, $\alpha^2=0$ for every $\alpha\in H^2(G/K)$, and $G/K$ is not symplectic.
\end{example}


The first proof of Lie 2 for groupoids and algebroids appeared in the appendix of \cite{mx}, as a tool to integrate the compatible Poisson bivector on a Lie bialgebroid, viewed as a Lie algebroid map. This idea was carried over in the context of symplectic groupoids by \cite{bco}, yielding a conceptual proof for the following well-known result, that we recall here before developing a Hausdorff version. 

\begin{proposition}\label{prop:application1}
Given $(M, \pi)$ a Poisson manifold, if the induced Lie algebroid $A\then M$ is integrable by a Lie groupoid $G\toto M$, then the source-simply connected integration $\tilde G\toto M$ inherits the structure of a (possible non-Hausdorff) symplectic groupoid.
\end{proposition}

\begin{proof}[Sketch of proof (following \cite{bco})]
The Lie equivalence $\tilde G\to G$ yields another $T\tilde G\to TG$. Since the source-fibers of $T\tilde G$ are affine bundles over the source-fibers of $\tilde G$, $T\tilde G$ is the source-simply connected integration of $TA$. Then the canonical symplectic structure on $A=T^*M$, regarded as an algebroid morphism $\omega^b_{can}: TA \to T^*A$, integrates to a groupoid morphism $T\tilde G\to T^*\tilde G$ (Proposition \ref{prop:lie2}). This turns out to be the multiplicative symplectic structure on $\tilde G$.
\end{proof}

%




To establish the Hausdorff version of the previous result, our first step is to show that the tangent of the maximal Hausdorff integration is the maximal Hausdorff integration of the tangent.

\begin{lemma}
\label{lemma:hatTG}
Given $G$ a Hausdorff groupoid with algebroid $A$, the Lie groupoid $T\hat G$ is the maximal Hausdorff integration of $TA$ over $TG$.
\end{lemma}

\begin{proof}
Let $\hat K$ be the kernel of $\tilde G\to\hat G$, which is the intersection of all the swind subgroupoids $M\subset K\subset\tilde K$.
The groupoid $T\hat G$ is Hausdorff and projects into $TG$ via a Lie equivalence. It remains to show that $T\hat G$ is maximal in the sense of Theorem \ref{thm:HLie1}, or equivalently, that the intersection of all the closed swind subgroupoids $M\subset K'\subset T\tilde K$ is exactly $T\hat K$. Given such a $K'$, it is open and closed in $T\tilde K$, and therefore, a union of connected components, hence equal to $TK$ for some subgroupoid $K$ of $\tilde K$. It is easy to check that this $K$ must be closed, smooth, wide, intransitive, normal, and with discrete isotropy. Then $\hat K\subset K$ and the result follows.
\end{proof}


The second step is a linear version of Proposition \ref{prop:subalgbd}, showing that VB-subalgebroids correspond to invariant linear foliations. Given $E\to M$ a vector bundle, we say that $F$ is a {\bf linear foliation} on $E$ if it is invariant under the multiplication by scalars. This is equivalent to saying that $F\then E$ is a VB-algebroid over $F_0\then M$, the foliation restricted to the zero section. Note that $\mon(F)$ is canonically a vector bundle over $\mon(F_0)$.

\begin{lemma}
\label{lemma:linear-subalgbd}
Given $\Gamma\toto E$ a VB-groupoid and $S\subset E$ a vector subbundle, there is a 1-1 correspondence between VB-subalgebroids $B\then S$ of $A_\Gamma\then E$ and $\Gamma$-invariant linear foliations $F$ on $t^{-1}(S)\subset \Gamma$.
\end{lemma}

\begin{proof}
It follows by combining Proposition \ref{prop:subalgbd} with the characterization of VB-groupoids and VB-algebroids as Lie groupoids and Lie algebroids endowed with a regular compatible action of the multiplicative monoid $(\R,\cdot)$ \cite{bcdh}.
\end{proof}


A peculiarity about linear foliations is that the holonomy at the zero section somehow controls the holonomy on the total space. Intuitively, if there is a loop $\gamma$ with non-trivial holonomy, then it has a transverse dislocation  $\gamma'$ that ceases to be a loop, and therefore, the paths $\epsilon\gamma'$ are transverse dislocations of $0\gamma$ which are not loops, proving that $0\gamma$ has also holonomy. We give a concise proof using our algebraic approach to holonomy.

\begin{lemma}\label{lemma:linear-foliations}
Let $E\to M$ be a vector bundle and $F\subset TE$ a linear foliation. If $\hol_x(F)=0$ for some $x\in M$ then $\hol_e(F)=0$ for every $e\in E_x$.
\end{lemma}

\begin{proof}
Since $\hol_x(F)=0$, we know that $s:I(\mon(F))\to M$ is locally bijective at $x=\id_x$. Let $e\xfrom g e$ be a loop at $e\in E_x$, and let $x\xfrom{0g}x$ its projection on the zero section. Since $0g$ has no holonomy, we know that $s:I(\mon(F))\to E$ is locally bijective at $0g$. But this is a vector bundle map over $s:I(\mon(F_0))\to M$, and therefore, it has to be a linear isomorphism between the fibers $I(\mon(F))_{0g}\to E_x$ and locally bijective at $0g$ in the base. This proves that the vector bundle map is also locally bijective around $g$ and at any other point over $0g$.
\end{proof}


We are ready to present the main result of the section, roughly saying that if a Poisson manifold is integrable by a Hausdorff groupoid, then it is also integrable by a Hausdorff symplectic one.

\begin{theorem}\label{thm:application}
Given $(M, \pi)$ a Poisson manifold, if the induced Lie algebroid $A\then M$ is integrable by a Hausdorff groupoid $G\toto M$, then $\hat G\toto M$ is a Hausdorff symplectic groupoid.
\end{theorem}

\begin{proof}
We may suppose $G=\hat G$. We want to integrate the isomorphism of Lie algebroids
$\varphi=\omega_{can}^b:(TA\then TM)\to (T^*A\then A^*)$ 
to an isomorphism of Lie groupoids $TG\to T^*G$ defining a multiplicative symplectic form on $G$. 
Writing $B\then S$ for the graph of $\varphi$, which is a VB-subalgebroid of $TA\times T^*A\then TM\times A^*$, by Theorem \ref{thm:HLie2v1} and Lemma \ref{lemma:hatTG}, we just need to show that the foliation $F^\varphi$ on $(t\times t)^{-1}(S)\subset TG\times T^*G$ has no holonomy. 

Since $B\to A_{TG}\times A_{T^*G}\to A_G$ is an isomorphism, the map $\pi:(t\times t)^{-1}(S)\subset TG\times T^*G\to TG$ is \'etale when restricted to each leaf of $F^\varphi$, so the fibers of $\pi$ are transversal to $F^\varphi$, and it defines an Ehresmann connection. To see that $F^\varphi$ has no holonomy it is enough to show that if $\gamma$ is a horizontal loop for that connection, then evey horizontal lift of $\pi\gamma$ is also a loop.
The foliation $F^\varphi$ is linear by Lemma \ref{lemma:linear-subalgbd}, so we can suppose that $\gamma$ is in the zero section by Lemma \ref{lemma:linear-foliations}.


It is convenient to recall the exact sequence 
$0 \to A_y\xto{d R_g}T_gG \xto{ds} T_xM\to 0$
from where the formulas for the source and target maps of $TG$ and $T^*G$ can be derived \cite{bcdh}.
Given $y\xfrom g x$ and $y\xfrom h z$ in $G$, and given $v\in T_gG$ and $\alpha\in T^*_hG$, it follows that 
$$
(v,\alpha)\in (t\times t)^{-1}(S)
\iff
\varphi(dt(v))=(dR_h)^*(\alpha\bar_{G(-,z)})\in A_y^*.$$ 
Then a curve $\gamma(r)=(g_r, v_r, h_r, \alpha_r)\in (t\times t)^{-1}(S)$ is in a leaf of $F^\varphi$, or in other words, it is horizontal for the Ehresmann connection, if and only if for every $r_0$ we have
$$\varphi\left[
\left.\frac{d}{dr}\right\rvert_{r=r_0}(g_r,v_r)(g_{r_0},v_{r_0})^{-1}
\right]
=
\left.\frac{d}{dr}\right\rvert_{r=r_0}(h_r,\alpha_r)(h_{r_0},\alpha_{r_0})^{-1}
$$
where the multiplications are in $TG$ and $T^*G$, respectively.

So let $\gamma$ be a horizontal loop such that $\gamma(r)=(g_r, 0, h_r, 0 )$, namely the loop sits within the zero section. If $\gamma'$ is another horizontal lift for $\pi\gamma$, then $\gamma'(0)=(g_r,0,h'_0,\alpha_0')$, and the uniqueness of solutions in the above differential equation readily implies
$$h'_r=g_rg_0^{-1}h'_0 \qquad \alpha'_r=(0_{h'_r(h'_{r_0})^{-1}})\alpha'_0.$$
Note that the product $0_{h'_r}\alpha'_0$ in $T^*G$ makes sense because 
$t(\alpha'_0)=(dR_h)^*(\alpha'_0\bar_{G(-,z)})=\varphi(dt(v_0))$ and $v_0=0$. 
Finally, since $\gamma(1)=\gamma(0)$ we have that $g_1=g_0$, and this applied to the explicit formulas for $\gamma'$ show that $\gamma'(1)=\gamma'(0)$, namely that $\gamma'$ is also a loop, so $F^\varphi$ has no holonomy.

It is straightforward to check that the integrated morphism $\phi:TG\to T^*G$ is indeed $\omega^b$ where $\omega$ is a multiplicative symplectic form on $G$.
\end{proof}

\begin{corollary}\label{cor:application}
A Poisson manifold with a Hausdorff integration admits a Hausdorff complete symplectic realization.
\end{corollary}

\begin{proof}
Let $(M,\pi)$ be the Poisson manifold, $A\then M$ its induced algebroid, and $G\toto M$ a Hausdorff integration. Then $\hat G\toto M$ is a Hausdorff symplectic groupoid, and therefore, $s:\hat G\to M$ is a Hausdorff complete symplectic realization \cite{cf2}.
\end{proof}

We conjecture that a Hausdorff symplectic complete realization gives rise to a Hausdorff symplectic groupoid. This converse of the previous corollary would give a Hausdorff version of \cite[Thm 8]{cf2}.  We plan to address this problem elsewhere.


\appendix

\section{Quotients of (non-Hausdorff) manifolds}


We include in this appendix a proof of the Godement criterion for quotients of smooth manifolds, with emphasis in the non-Hausdorff case, and a lemma on Lie groupoid actions which is used along the paper.
Our treatment is alternative and complementary to the Hausdorff version of \cite[\S 9]{rui} and the analytic version in \cite[II.3.12]{serre}.

Our manifolds $M$ are allowed to be non-Hausdorff, unless otherwise specified.
Our main tool in proving the criterion is the construction of fibered products between transverse smooth maps.



\begin{lemma}\label{lemma:pull-back}
Let $f_1:M_1\to N$ and $f_2:M_2\to N$ be transverse smooth maps. Then the fibered product $M_1\times_N M_2\subset M_1\times M_2$ is embedded, and for every $(x_1,x_2)\in M_1\times_N M_2$ the following sequence is exact:
$$0\to T_{(x_1,x_2)}(M_1\times_N M_2)\to T_{x_1}M_1\times T_{x_2}M_2\xto{df_1\pi_1-df_2\pi_2} T_xN\to 0.$$
Moreover, if $f_1$ is an embedding, then so does its base-change $M_1\times_N M_2\to M_2$.
\end{lemma}
\begin{proof}
Given a local chart $U\xto{\varphi} \R^n$ of $N$, let $V=f_1^{-1}(U)\times f_2^{-1}(U)$, and consider the function $F:V\to \R^n$ given by $F(x_1, x_2)=\varphi(f_1(x_1))-\varphi(f_2(x_2))$. Observe that $0$ is a regular value of $F$, for $dF=d\varphi (df_1\pi_1-df_2\pi_2)$ and $f_1 \pitchfork f_2$. It follows from the constant rank theorem that $F^{-1}(0)=(M_1\times_N M_2)\cap V\subset V$ is embedded with tangent space $\ker dF=\ker (df_1\pi_1-df_2\pi_2)$.

Suppose now that $f_1$ is an embedding. 
Working locally, we can assume that $M_1=\R^p$, $N=\R^{p+q}$ and that $f_1$ is just the inclusion $\R^p\to \R^p\times\R^q$, $x\mapsto (x,0)$.
Then $M_1\times_N M_2$ identifies with the preimage of $0$ along $\pi_2 f_2:M_2\to\R^q$. This composition is a submersion by the transversality hypothesis. The result follows by the constant rank theorem.
\end{proof}

Given $M$ a manifold, possibly non-Hausdorff, and $R\subset M\times M$ an equivalence relation, if the quotient $M/R$ admits a manifold structure so that  the projection $\pi:M\to M/R$ is a submersion, then it easily follows from \ref{lemma:pull-back} that $R=M\times_{M/R}M\subset M\times M$ is an embedded submanifold and that the projection $\pi_1:R\to M$ is a surjective submersion.  
It turns out that these necessary conditions for $R$ are also sufficient to garantee that $M/R$ is indeed a manifold. 

\begin{proposition}[Godement criterion]\label{prop:godement}
Let $M$ be a manifold and $R\subset M\times M$ be an equivalence relation that is an embedded submanifold and makes $\pi_2:R\to M$ a submersion. Then $M/R$ inherits a unique canonical smooth structure such that the projection $\pi:M\to M/R$ is a submersion. Moreover, $M/R$ is Hausdorff if and only if $R\subset M\times M$ is closed.
\end{proposition}

\begin{proof} 
{\em Step 1: Describing the orbits.}
Given $x\in M$, write $O_x$ for its equivalence class or orbit. 
It follows from the following fibered product diagram and Lemma \ref{lemma:pull-back} that $O_x\subset M$ is an embedded submanifold with tangent space given by 
$T_yO_x\times 0\cong T_{(y,x)}R  \cap (T_yM\times 0)$:
		$$\xymatrix{
			O_x\times \{x\} \ar[r]^{} \ar[d]_{}  & \{x\} \ar[d]  \\ 
			R \ar[r]^{\pi_2} & M
		}$$

{\em Step 2: Building the charts.}
Let $S_x\subset M$ be a ball-like submanifold through $x$ such that $T_xS_x\oplus T_xO_x=T_xM$. After eventually shrinking $S_x$ we can assume that (i) $T_yS_x\oplus T_yO_y=T_y M$ for all $y\in S_x$, and (ii) $\pi:S_x\to M/R$ is injective. We then use $\alpha_x:S_x\to M/R$ to define a chart. Note that (i) is equivalent to $T_yS_x\times 0\cap T_{(y,y)}R=0$, which is an open condition on $y$. Regarding (ii), 
the inclusion $S_x\times S_x\subset M\times M$ is transverse to $R\subset M\times M$, because of (i) and because $O_y\times O_y\subset R$ for every $O_y$. By Lemma \ref{lemma:pull-back} $P= (S_x\times S_x)\cap R\subset S_x\times S_x$ is an embedded submanifold and $\dim P=\dim S_x$:
		$$\xymatrix{
			P \ar[r]^{} \ar[d]  & R \ar[d]^{}  \\ 
			S_x\times S_x \ar[r]^{} & M\times M.
		}$$
The inclusion $\Delta(S_x)\subset P$ must be an open embedding, then we can find a ball-like open $x\in U$ such that $U\times U\cap P\subset \Delta(S_x)$ and $U\to M/R$ is injective.

{\em Step 3: Compatibility between charts.}
Given two charts $S_x\to M/R$ and $S_y\to M/R$ with nontrivial intersection, 
the inclusion $S_y\times S_x\subset M\times M$ is transverse to $R$, by an argument analogous to the one in Step 2. Then by Lemma \ref{lemma:pull-back} $P=(S_y\times S_x)\cap R\subset S_x\times S_y$ is an embedded submanifold with $\dim P=\dim S_x$.
The projections $\psi_1:P\to S_y$, $\psi_2:P\to S_x$, are injective, and the transition between the two charts is given by the composition $\psi_1\psi_2^{-1}$, so it is enough to show that $\psi_2:P\to S_x$ is 
\'etale, namely a local diffeomorphism. 
By Lemma \ref{lemma:pull-back} 
$T_{(y',x')}P=(T_{y'}S_y\times T_{x'}S_x)\cap T_{(y',x')}R$, 
and by Step 1 $\ker d\psi_2=T_{(y',x')}R\cap (T_{y'}M\times 0)=T_{y'}O_{y'}\times 0$. Then $d_{(y',x')}\psi_2:T_{(y',x')}P\to T_{x'}S_x$ is injective, hence an isomorphism, and $\psi_2$ is \'etale.

{\em Step 4: $\pi$ is a smooth submersion.}
By the Step 1 the intersection $(S_x\times M)\cap R$ is an embedded submanifold of $M\times M$ of dimension $\dim M$, and $\pi_2:(S_x\times M)\cap R\to M$ has injective differential, hence it is locally invertible. If $\phi:U\to (S_x\times M)\cap R$ is a local inverse around $x$, then $\pi_1\phi=\alpha_x^{-1}\pi:U\to S_x$, and therefore $\pi:M\to M/R$ is smooth. It is also a submersion because it admits local sections induced by the inclusions $S_x\to M$. This implies that $\pi$ is an open map, so the topology induced by our charts is indeed the quotient topology, and that the smooth structure on $M/R$ is uniquely determined by that on $M$.

{\em Step 5: Hausdorffness.}
If $M/R$ is Hausdorff then $R=(\pi\times\pi)^{-1}(\Delta_{M/R})\subset M\times M$ is closed. Conversely, if $R$ is closed, given $(y,x)\notin R$, we can find a basic open $(y,x)\in V\times U\subset (M\times M)\setminus R$, and since $\pi:M\to M/R$ is open, $\pi(U)$ and $\pi(V)$ separate $\pi(x)$ and $\pi(y)$ in $M/R$.
\end{proof}

All the relations we consider in the paper arise from actions of Lie groupoids. 
Given $G\toto M$ a Lie groupoid, and given $E$ a possibly non-Hausdorff manifold, a (right) {\bf action} $E\curvearrowleft G:\rho$ over $\mu:E\to M$ is a map $\rho:E\times_M G\to E$, $(e,g)\mapsto eg$, defined on the fibered product between $\mu$ and $t$, such that $\mu(eg)=s(g)$, $\rho_{hg}=\rho_h\rho_g$ and $\rho_{1_x}=\id_{E_x}$ (see \cite{dh} for more details). We say that $\rho$ is a {\bf principal} action if the {\it anchor} map $(\pi_1,\rho):E\times_M G\to E\times E$ is an embedding. Note that the anchor is injective if and only if the action is free, namely if $eg=e$ implies that $g=1_{\mu{e}}$, and that it is a topological embedding if and only if the division map $\delta:R\to G$, $(e,eg)\mapsto g$ is continuous, where $R$ is the equivalence relation. It follows from Godement criterion \ref{prop:godement} that the orbit space of a principal action is a well-defined possibly non-Hausdorff manifold. 


\begin{lemma}\label{principal}
\begin{enumerate}[(a)]
\item If $K\subset G\toto M$ is a wide embedded subgroupoid then right multiplication $G\times_M K\to G$, $(g,k)\mapsto gk$, is a principal action.
\item If $E\curvearrowleft G:\rho$ is principal, $E'\curvearrowleft G:\rho'$ is some other action and $\phi:E'\to E$ is equivariant, 
namely $\mu\phi=\mu'$ and $\phi(e'g)=\phi(e')g$ for every $e',g$,
then $\rho'$ is also principal.  
\end{enumerate}
\end{lemma}

\begin{proof}
Regarding (a), note that the image of the anchor map $(g,k)\mapsto (g,gk)$ is included in the fibered product $G {_t\times_t} G$, so we can compose it with the division map $G {_t\times_t} G\to G\times G$, $(g,h)\mapsto (g,g^{-1}h)$, and recover the canonical embedding $G\times_M K\to G\times G$.

To prove (b) we will show that the anchor map $(\pi_1,\rho')$ is injective, immersive and a topological embedding. If $e'g=e'$ then $\phi(e')g=\phi(e')$, and since $\rho$ is free, $g$ must be a unit, proving that $\rho'$ is also free, and $(\pi_1,\rho')$ injective. Now let $(w,v)\in T_{(e',g)}(E'\times_M G)\subset T_{e'}E'\times T_gG$ such that $d(\pi_1,\rho')(w,v)=(w,d\rho'(w,v))=0$. Then $w=0$ and $0=d\phi d\rho'(0,v)=d\rho(0,v)$. Since $(\pi_1,\rho)$ is immersive, $v=0$ and $(\pi_1,\rho')$ is also immersive. Finally, calling $R'$ and $R$ the relations defined by $\rho'$ and $\rho$, we can write the division map $\delta'$ as the composition \smash{$R'\xto{\phi\times\phi} R\xto\delta G$}, hence $\delta'$ is continuous and the anchor is a topological embedding.
\end{proof}




\frenchspacing

{\footnotesize

\bigskip
 
\sf{\noindent Matias del Hoyo\\
Universidade Federal Fluminense (UFF),\\ 
Rua Professor Marcos Waldemar de Freitas Reis, s/n,
Niterói, 24.210-201 RJ, Brazil.\\
mldelhoyo@id.uff.br}

\

\sf{\noindent Daniel Lopez\\
Instituto de Matematica Pura e Aplicada (IMPA),  \\ 
Estrada Dona Castorina 110, Rio de Janeiro, 22460-320, RJ, Brazil.\\
daflopez@impa.br}

}

\end{document}